\documentclass[a4paper,11pt]{article}
\usepackage{amstext}
\usepackage{amsmath}
\usepackage{amssymb}
\usepackage{amscd}
\usepackage{amsfonts}
\usepackage{amsthm}
\usepackage{parskip}
\usepackage{fancyhdr}

\usepackage{graphicx}
\usepackage{latexsym}
\usepackage[active]{srcltx}

\usepackage{epsfig}
\usepackage{tikz}
\usetikzlibrary{matrix,arrows,decorations.pathmorphing}
\usepackage{anysize}

\input xy
\xyoption{all}
\title{Transitivity of codimension one non-invertible conservative skew-products}
\author{Martin Andersson \\ Universidade Federal Fluminense \\ nilsmartin@id.uff.br 
\and Javier Correa \thanks{The second author has been supported by CAPES} \\ Universidade Federal do Rio de Janeiro \\ jacorrea88@gmail.com}

\begin{document}

\newtheoremstyle{basico}% name of the style to be used
  {0,2cm}% measure of space to leave above the theorem. E.g.: 3pt
  {}% measure of space to leave below the theorem. E.g.: 3pt
  {\itshape}% name of font to use in the body of the theorem
  {0,5cm}% measure of space to indent
  {\bfseries}% name of head font
  {}% punctuation between head and body
  {0,2cm}% space after theorem head; " " = normal interword space
  {\thmname{#1}\thmnumber{ #2}:\thmnote{ #3}}% Manually specify head
\theoremstyle{basico}  
  
\newtheorem{teo}{Theorem}[section]
\newtheorem{coroprin}{Corollary}
\newtheorem{lemaprin}{Lemma}
\newtheorem{propprin}{Proposition}
\newtheorem{lema}[teo]{Lemma}
\newtheorem{prop}[teo]{Proposition}
\newtheorem{obs}[equation]{Observation}
\newtheorem{ejem}{Example}
\newtheorem{defn}{Definition}
\newtheorem{obsnot}{Notation}
\newtheorem{afim}{Afirmation}
\newtheorem{teoprin}{Theorem}
\newtheorem{quest}{Question}
\newtheorem{rem}[teo]{Remark}

\newcommand{\noi}{\noindent}

\newcommand{\La}{\Lambda}
\newcommand{\s}{\sigma}
\newcommand{\al}{\alpha}
\newcommand{\de}{\delta}
\newcommand{\e}{\epsilon}
\newcommand{\ga}{\gamma}
\newcommand{\be}{\beta}

\newcommand{\la}{\lambda}
\newcommand{\N}{\mathbb{N}}
\newcommand{\Z}{\mathbb{Z}}
\newcommand{\Q}{\mathbb{Q}}
\newcommand{\R}{\mathbb{R}}
\newcommand{\A}{\mathcal{A}}
\newcommand{\B}{\mathcal{B}}
\newcommand{\U}{\mathcal{U}}
\newcommand{\Pa}{\mathcal{P}}
\newcommand{\T}{\mathbb{T}}
\newcommand{\C}{\mathbb{C}}
\newcommand{\hf}{\hat{f}}
\newcommand{\hu}{\hat{U}}
\newcommand{\hv}{\hat{V}}
\newcommand{\htag}{\texttt{\#}}
\newcommand{\Ker}{\text{Ker}}
\newcommand{\Ima}{\text{Im}}

\maketitle

\abstract{In this work we explore the problem of transitivity of volume preserving skew-products endomorphisms of the n-torus. More specifically, we establish relationships between transitivity and the action induced by the skew-product in the fundamental group.}

\marginsize{2,5cm}{2,5cm}{2cm}{3cm}
\setlength{\parskip}{0,2cm}
\setlength{\parindent}{0.5cm}

\section{Introduction}
In dynamical systems, an important family to study is the family of skew-products. They are easy to build and have a simple structure, yet they have enough complexity to model more general systems. Our focus in this paper will be volume-preserving non-invertible skew-products. A general goal for volume-preserving maps is to know whether or not they are ergodic. Since ergodicity is stronger than transitivity, we consider a good starting point to address the transitivity.

By a toral endomorphism we mean a surjective local homeomorphism $f:\T^n\to\T^n$. In other words, a covering map from $\T^n$ to itself. Let $\mu$ be the Haar measure on $\T^n$. We say that $f$ is \emph{volume-preserving} (or \emph{conservative}) if $\mu(f^{-1}(B)) = \mu(B)$ for every Borel measurable set $B\subset \T^n$. 

We say that $f$ is transitive if there exists $z\in \T^n$ such that $\T^n = \overline{\{f^n(z):n\in \N\}}$.

It is reasonable to expect transitivity for volume-preserving non-invertible endomorphisms  under quite general circumstances. First of all, note that linear (hence volume-preserving) non-invertible toral endomorphisms are always transitive (in fact ergodic \cite{AH94}). Indeed, they are robustly transitive: every $C^1$ close endomorphism (not necessarily conservative) is also transitive. In dimension two, every conservative endomorphism homotopic to a non-invertible hyperbolic linear map is transitive \cite{An15}.  Furthermore, Lizana and Pujals in \cite{LiPu12}  provided  sufficient conditions for $C^1$ endomorphisms to be robustly transitive. Rather than dealing with conservative endomorphism, they consider endomorphisms with Jacobian larger than one.

Given $h:\T^{n-1}\to \T^{n-1}$ and $g:\T^{n-1}\times \T^1\to \T^1$ we define $f:\T^n\to \T^n$ by $f(x,t)=(h(x),g(x,t))\ \forall x\in \T^{n-1},\ \forall t\in \T^1$. We say that $f$ is a skew-product of codimension $1$ and has the form $f = (h,g)$. We shall refer to $h$ as the action in the base and $g$ as the action in the fibers. For $x\in \T^{n-1}$ let us define the map $g_x:\T^1\to\T^1$ by $g_x(t) = g(x,t)$. Note that, since $f$ is a covering map, so are $h$ and $g_x$ for every $x \in \T^{n-1}$. Let $\deg(f)$, $\deg(h)$, and $\deg(g_x)$ denote their (unsigned) degrees, this is, the number of preimages of any point. Since $g$ is continuous and $\deg(g_x)$ is a homotopy invariant,  the number $\deg(g_x)$ does not depend on $x$ and we denote it by $\deg(g)$. Observe that $\deg(f) = \deg(g)\deg(h)$.

A classical family of invertible skew-products are maps $f:\T^2\to \T^2$ of the form $f(x,t)=(x+\alpha, t + \phi(x))$, where $\alpha\in \R$ is an irrational number and $\phi:\T^1\to \T^1$ is a continuous map. If the cohomological equation
\[u(x+\alpha) - u(x)= \phi(x),
\]
has a continuous solution, then $f$ is conjugated to $f_0(x,t)=(x+\alpha, t)$ and therefore it is not transitive. This is an example of a non transitive skew-product where $\deg(f) = \deg(h) = \deg(g) = 1$. Similar examples can be constructed by replacing the base map $x \mapsto x+\alpha$ by $x \mapsto kx \mod 1$. In this case we have $\deg(f) = \deg(h) = |k|$ and $\deg(g)=1$.

We would like to address now what happens when $\deg(g) \geq 2$. Observe that if $f=(h,g)$ preserves the Haar measure on $\T^n$, then $h$ preserves the Haar measure on $\T^{n-1}$. Moreover, if $f$ is transitive, so is $h$.

In order to announce the main theorem of this article we will need to define a linear map associated to a torus endomorphism.  Given $f:M\to M$, let $f_\htag:\pi_1(M)\to \pi_1(M)$ be the induced morphism on the fundamental group $\pi_1(M)$ of $M$. If we take $M=\T^n$, then $\pi_1(\T^n)$ is isomorphic to $\Z^n$ and we can represent $f_\htag $  by a linear matrix $A_f\in M_n(\Z)$.  We shall often refer to the matrix $A_f$ (or the maps it defines on $\R^n$ and $\T^n$) as the \emph{linear part of $f$}.
Observe that if $f$ is a skew-product, we have $f(\{x\}\times \T^1) = \{h(x)\}\times \T^1$ for every $x\in\T^{n-1}$. Therefore the vector $e_n = (0,\dots,0,1)$ is an eigenvector of $A_f$. It is not hard to see that the eigenvalue associated to $e_n$ is either $\deg(g)$ or $-\deg(g)$. For $A_f$, we consider  its Jordan normal form $J$ and  the Jordan block $J_n$ associated to the eigenvector $e_n$. 
 
The main Theorem is the following:

\begin{teoprin}\label{SP2}
Let $f:\T^n\to \T^n$ be a skew-product endomorphism of codimension 1 of the form $f=(h,g)$, with $h: \T^{n-1} \to \T^{n-1}$ transitive and $\deg(g) \geq 2$. If $f$ is volume-preserving and $\dim(J_n)=1$, then $f$ is transitive. 
\end{teoprin}

We emphasize that our result is purely topological, this is, it does not rely on any $C^r$ regularity of the maps $f$, $h$ and $g$ and, in particular, does not make use of any hyperbolic structure. Neither do our proofs require the density (not even the existence) of periodic points.

Before  discussing the hypothesis $\dim(J_n)=1$ in more detail we would like to point out some particular cases in which it holds. Suppose that $f=(h,g)$ is a volume-preserving skew-product with $h$ transitive, $\deg(g)\geq 2$ and $\deg(h)=1$. The hypothesis on the degree of $h$ means that it is a homeomorphism, and by a previous observation, a volume-preserving homeomorphism. It is not hard to see then, that at least in the $C^1$ case, each of the maps $g_x: \T^1 \to \T^1$ is uniformly expanding. (In the $C^0$ case a variant of uniform expansion occurs.) Assuming transitivity of $h$ allows us to easily conclude that $f$ itself is transitive.  This proof is unrelated to (and indeed much easier than) the proof of Theorem \ref{SP2}. We therefore state it separately:

\begin{teoprin}\label{SP1} 

Let $f:\T^n\to \T^n$ be a skew-product of codimension 1 of the form $f=(h,g)$. If $f$ is a volume-preserving endomorphism, $h$ is a transitive homeomorphism, and \linebreak $|\deg(g)|\geq 2$, then $f$ is transitive. 
\end{teoprin}

At a first glance, one could imagine that it would be easier to obtain transitivity in the case where $\deg(h)\geq 2$, due to the extra complexity coming from the base. But that is not the case, because when $\deg(h)\geq 2$, the condition of $f$ being volume-preserving does not imply that the  $g_x$  have to be uniformly expanding. 

\pagebreak

\begin{teoprin}\label{Examp}
Given $n \geq 2$, and $k\geq 2$ there exists a volume-preserving skew-product endomorphism $f:\T^n \to \T^n$ of codimension 1 of the form $f=(h,g)$ with $\deg(h)\geq 2$ and $\deg(g) = k$ such that
\begin{itemize}
\item $h$ is transitive,
\item there are a fixed point $x_0$ of $h$, and an interval $I \subset \T^1$ such that $g_{x_0}$ is uniformly contracting on $I$,
\item the linear part of $f$ is given by the matrix
\[A_f= \left( \begin{array}{cc} 
2  Id & \begin{array}{c}0\\ \vdots \\ 0 \end{array} \\
\begin{array}{ccc}1 & \dots & 1\end{array}  & k 
\end{array}\right),\]
where $k$ is the eigenvalue associated to $e_n$.
\end{itemize}
\end{teoprin}

Note that if we take $k>2$, then the form of the linear part of $f$ in Theorem \ref{Examp} implies that $f$ satisfies the hypothesis of Theorem \ref{SP2} and is therefore transitive.

Theorem \ref{Examp} suggests that, in order to deal with transitivity in the case where $|\deg(h)|\geq 2$, one has to adopt global arguments that make use of the way that $f$ wraps curves around the manifold rather than localized behavior such as expansion or contraction near a given point.

 Let us now give some examples where the hypothesis $\dim(J_n)=1$ holds, obtaining some corollaries of Theorem \ref{SP2}:

\begin{coroprin}\label{CorDom1}
Let $f:\T^n\to \T^n$ be a skew-product of codimension 1 with the form $f=(h,g)$. Suppose that $f$ is a volume-preserving endomorphism, $h$ is a transitive endomorphism, and $1\leq \deg(h)<\deg(g)$. Then, $f$ is transitive. 
\end{coroprin}

This implies that Theorem \ref{SP1} is really a corollary of Theorem \ref{SP2}. The result is also true for another type of domination:

\begin{coroprin}\label{CorDom2}
Let $f:\T^n\to \T^n$ be a skew-product of codimension 1 with the form $f=(h,g)$. If $f$ is a volume-preserving endomorphism, $h$ is a transitive endomorphism, $\deg(g)>1$ and $|A_h v|> \deg(g) |v|$ $\forall v\in\R^n -\{0\}$, then $f$ is transitive. 
\end{coroprin}

Finally if $A_f$ is diagonalizable, then all the Jordan blocks have dimension $1$ and therefore we have:

\begin{coroprin}
Let $f:\T^n\to \T^n$ be a skew-product of codimension 1 with the form $f=(h,g)$. If $f$ is a volume-preserving endomorphism, $h$ is a transitive endomorphism, $\deg(g)>1$,  and $A_f$ is diagonalizable, then $f$ is transitive. 
\end{coroprin}

The examples built in Theorem \ref{Examp} can verify the hypothesis of Theorem \ref{SP2} or the previous corollaries and therefore they would still be transitive. 

Let us give a sketch of the proof of Theorem \ref{SP2}:

We call an invariant region an open set which verifies $f^{-1}(U)= U$. If $f$ is a volume-preserving endomorphism, the lack of transitivity is equivalent to the existence of more than one invariant region (Check Proposition \ref{NotTran}). We start by studying the structure of the fundamental group of such invariant regions.
Our starting point is the set of  techniques used in \cite{An15}, where the first author proved that, given  a volume-preserving toral endomorphism $f:\T^2 \to \T^2$,  with $|\deg(f)|\geq 2$,  such that $A_f$ is hyperbolic, then $f$ is transitive. 

Using non-invertibility and the hypothesis  that $f$ is conservative one can prove that, if $i:U\to \T^n$ denotes the inclusion, then $i_\htag(\pi_1(U))$ is not trivial.  This is the main use we give to the volume-preserving hypothesis and the same results could be obtained by switching for the hypothesis $\Omega(f) = \T^n$ (where $\Omega(f)$ is the non-wandering set) as it is done by Ranter in \cite{Ra16}.  Our next step is to conclude that, not only is $i_\htag(\pi_1(U))$ non-trivial, but it has to be big enough such that the action of $f_{\htag|i_\htag(\pi_1(U))}$ has the same ``degree'' as $f_\htag$. After that, we take a lift from $f$ and using the skew-product structure  we construct an invariant hypersurface $S$. This hypersurface is obtained by the expansiveness of the linear part of $f$ along the fibers and the $\dim(J_n)=1$ hypothesis. In particular, the dynamic of $S$ is conjugated to the dynamic of $h$. By the previous arguments we prove that the lift of any invariant region intersects such hyper-surface and from the transitivity of $h$ we obtain a contradiction.

Let us observe that the hypothesis $\dim(J_n)=1$ is a necessary condition to imply the existence of the hypersurface and therefore essential to our proof, yet we do not know whether there exists a counter-example to Theorem \ref{SP2} if this hypothesis is removed.

In section 2 we proof Theorem \ref{SP1} and Theorem \ref{Examp}. In section 3 we develop the setting we will be working and prove the results from algebraic topology we will need. In section 4 we study some properties of invariant subspaces from integer matrices. In section 5 we prove Theorem \ref{SP2}. Observe that all the results stated in section 3 hold for toral endomorphisms, not just skew-products.

%%%%%%%%%%%%%%%%%%%%%%%%%%%%%%%%%%%%%%%%%%%%%%%%%%%%%%%%%%%%%%%%%%%%%%%%%%%%%%%%%%%%%%%%%%%%%%%%%%%%%%%%%%%%%%%%%%%%%%%%%%%%%%%%%%%%%%%%%%%%%%%%%%%%%%%%%%%%%%%%%%%%
%%%%%%%%%%%%%%%%%%%%%%%%%%%%%%%%%%%%%%%%%%%%%%%%%%%%%%%                  SECTION 2                       %%%%%%%%%%%%%%%%%%%%%%%%%%%%%%%%%%%%%%%%%%%%%%%%%%%%%%%%%%%
%%%%%%%%%%%%%%%%%%%%%%%%%%%%%%%%%%%%%%%%%%%%%%%%%%%%%%%%%%%%%%%%%%%%%%%%%%%%%%%%%%%%%%%%%%%%%%%%%%%%%%%%%%%%%%%%%%%%%%%%%%%%%%%%%%%%%%%%%%%%%%%%%%%%%%%%%%%%%%%%%%%%
\newpage

\section{Theorem \ref{SP1} and Theorem \ref{Examp}}                                                        
 
Let us see the proof of Theorem \ref{SP1}

\begin{proof}

Let $\nu$ be the Haar measure on $\T^{n-1}$ and let $\lambda$ be the Haar measure on $\T$. 
We shall  first show that if $f$ preserves $\mu$, then $h$ preserves $\nu$.  Let  $r_1:\T^n\to \T^{n-1}$ be the projection $r_1(x_1,\dots, x_{n-1},x_n) = (x_1,\dots, x_{n-1})$ $\forall (x_1,\dots,x_n)\in \T^n$.  Given $B\subset \T^{n-1}$, we have that $r_1^{-1}(B)  = B\times \T^1$. Thus ${r_1}_* \mu = \nu$. By the skew-product structure $f^{-1}(B\times \T^1)= h^{-1}(B)\times \T^1$. Since $f$ is volume-preserving:
\[ \nu(h^{-1}(B))= \mu (h^{-1}(B)\times \T) = \mu(f^{-1}(B\times \T^1))=\mu(B\times \T^1) = \nu(B).\]

The  next step is to show that if $h$ is a homeomorphism, then $g_x$ preserves $\lambda$ for every $x \in \T^{n-1}$. It is instructive to consider the case in which $f$ is of class $C^1$. In this case, $f$, $h$, and each of the $g_x$ have well defined Jacobians. Let us denote these by $J(f, \cdot)$, $J(h, \cdot)$, and $J(g_x, \cdot)$ respectively. Since $f$ preserves $\mu$, we must have
\[1 = \sum_{(y,s)\in f^{-1}(x,t)} \frac{1}{J(f,(y,s))}= \sum_{y\in h^{-1}(x)} \frac{1}{J(h,y)}\sum_{s\in g_y^{-1}(t)} \frac{1}{J(g_y,s)}\ \forall x\in \T^{n-1},\ \forall t\in \T^1.\] 

Since $h$  is a homeomorphism, $\# h^{-1}(x) = 1$ and since it is volume-preserving, $J(h,y) = 1$. If we combine this with the previous equation, we obtain that:
\[\sum_{s\in g_y^{-1}(t)} \frac{1}{J(g_y,s)} = 1, \]
where $y = h^{-1}(x)$. Since $|\deg(g)|\geq 2$ and $J(g_y,s) > 0$, we can conclude that $J(g_y,s) > 1$ $\forall y\in \T^{n-1}$, $\forall s\in\T^1$. By continuity of $dg$ and compactness of $\T^n$, we have that $J(g_y,s) > 1+\epsilon$ for some $\epsilon > 0$. This implies that $g$  is expanding in the fibers.

In particular, given $I\subset \T^1$ there exists $k=k(I)>0$ such that for all $x\in \T^{n-1}$, $f^n(\{x\}\times I) = \{h^k(x)\}\times \T^1$. Let us take $U_1$ and $U_2$ open neighborhoods of $\T^{n-1}$, and $I_1$ and $I_2$ open neighborhoods of $\T^1$. We want to prove that $f^{k_1}(U_1\times I_1)\cap U_2\times I_2\neq \emptyset$ for some $k_1>0$. Taking $k$ associated to $I_1$ and using the transitivity of $h$, there exists $k_1 > k$ such that $h^{k_1}(U_1)\cap V_1 \neq \emptyset$. Then, $(h^{k_1}(U_1)\cap V_1)\times I_2 \subset f^{k_1}(U_1\times I_1)\cap U_2\times I_2$.

Now let us consider the more general case in which $f$ is only assumed to be a continuous surjective local homeomorphism. Our first assertion is that each $g_x$ preserves $\lambda$.

For the purpose of contradiction, suppose there is some $x$ such that $g_x$ does not preserve $\lambda$. That is equivalent to say that there is some continuous function $\phi: \T \to \R$ such that 
\begin{equation} \label{fiberinequality}
\int \phi \circ g_x \ d\lambda < \int \phi \ d\lambda.
\end{equation}
Since the map $\T^{n-1} \ni x \mapsto g_x \in C^0(\T, \T)$ is continuous, if (\ref{fiberinequality}) holds for some $x$, then it holds in an open set $U \in \T^{n-1}$. Let $\psi: \T^{n-1} \to \R$ be a non-negative continuous function, supported in $U$, such that $\int \psi \ d\nu>0$, and let $\varphi:\T^n \to \R$ be defined by $\varphi(x,t)=\phi(x)\psi(t)$. We claim that $\int \varphi \circ f \ d\mu < \int \varphi \ d\mu$, contradicting the $f$-invariance of $\mu$.

Indeed,
\begin{align}
\int \varphi \circ f \ d\mu  
& = \int \left( \int \phi(h(x)) \psi(g_x(t)) \ d\lambda(t) \right)  d\nu(x) \\
&= \int \phi(h(x)) \left(\int \psi \circ g_x \ d\lambda \right) d\nu(x) \\
&< \int \phi \circ h \ d\mu  \int \psi \ d\lambda \\
&= \int \phi \ d\nu \int \psi d\lambda = \int \varphi d\mu,
\end{align}
and we have arrived at the desired contradiction.

Now, since $g_x$ is not (necessarily) of class $C^1$, there may not exist $\epsilon>0$ such that $\lambda(g_x(I))>(1+\epsilon) \lambda(I)$ for every interval $I \subset \T$ such that $g_x:I \to g_x(I)$ is a homeomorphism. 
However, $\lambda(g_x(I))$ is always larger than $\lambda(I)$ so, by compactness, given any $K>0$ there is a $\delta>0$ such that if $\lambda(I)\geq K$ and $g_x$ is a homeomorphism from $I$ onto its image, then $\lambda(g_x(I))\geq \lambda(I)+\delta$. 
Writing $I_n = g_{h^{n-1}}(x) \circ \ldots g_{h(x)}\circ g_x(I)$, $n\geq 0$, then we see by induction that $\lambda(I_k)\geq \lambda(I)+ k \delta$ as long as $g_{h^k(x)}:I_k \to I_{k+1}$ is a homeomorphism. 
But $\lambda(I)+ k \delta$ is larger than $1$ for $k$ sufficiently large, so there must be some $k$ such that $I_k = \T$. Now we may apply the same argument as in the $C^1$ case to conclude that $f$ is transitive.
\end{proof}

Observe that when $|\deg(h)|\geq 2$,  we no longer have the condition $J(h,y)=1$ in the $C^1$ case. Instead, it is replaced by $\sum_{y\in h^{-1}(x)} \frac{1}{J(h,y)} = 1$. Since the sum has more than one term, this will imply  that $J(h,x) > 1$ for every $ x\in \T^{n-1}$, this is that $h$ expands volume on sufficiently small sets. But $J(h,\cdot)$ does not have to be constant, since a lesser volume expansion on some point $x_1$ can be compensated by a greater volume expansion on a point $x_2$, where $x_1$ and $x_2$ have the same image under $h$. This flexibility makes it possible to have a volume-preserving skew product which is contracting on some of its fibers. This is the content of Theorem \ref{Examp}. In particular, proving Theorem \ref{SP2} will require an entirely different approach than that in Theorem \ref{SP1}.

\begin{proof}[Proof of Theorem \ref{Examp}]
The example we are going to build will be piecewise linear and therefore $C^1$ in an open and dense set with full measure. The volume-preserving property will then be guaranteed by making sure that the equation
\begin{equation}\label{EqJac}
 \sum_{(x,t)\in f^{-1}(y,s)} \frac{1}{J(f,(x,t))} = 1
\end{equation}
hold for almost every point $(x,t)$ in $\T^{n-1}\times \T$. 

 Let $m=n-1$ denote the dimension of the base. For the base map $h: \T^m \to \T^m$ we take the linear endomorphism induced by the matrix $A = 2 \cdot Id$, where $Id$ is the identity matrix of size $m \times m$. Note that   $\deg(h) = |\det(A)| = 2^m$.  By standard arguments, $h$ is transitive. 
 
The action in the fibers will have two degrees of freedom in its construction. The first one is going to be the degree, denoted by $k = \deg(g)\geq 2$. The second one is going to be the rate of contraction $\lambda \in (0,1)$. In our  construction we are going to need $\lambda\in (1/2,1)$. We define the map $\phi:\T^1\to \T^1$ by
 \[\phi(t) =  \left\{
	\begin{array}{lll}
			 \lambda t & \mbox{if } t\in[0,1/(2\lambda)] \\
			\frac{(2k-1)\lambda}{2\lambda - 1}(t-1/(2\lambda)) + 1/2\ \mod\ 1& \mbox{if } t\in [1/(2\lambda),1]
		\end{array}
	\right.
\]

\pagebreak

Let us observe the following:
\begin{itemize}
\item $\lambda > 1/2$, so $2\lambda-1 > 1$. 
\item $\phi$ is clearly continuous at $1/(2\lambda)$ and, to check the continuity at $0$,  observe that  $\frac{(2k-1)\lambda}{2\lambda - 1}(1 - 1/(2\lambda)) + 1/2= k\in \N$.
\item $\phi$ contracts by the rate $\lambda$ the interval $[0,1/(2\lambda)]$ and expands by the rate $\eta = \frac{(2k-1)\lambda}{2\lambda - 1}$ the interval $[1/(2\lambda),1].$
\end{itemize}

We define  $g:\T^m \times \T^1\to \T^1$ by $g(x_1\dots,x_m,t) = x_1 +  \dots + x_m + \phi(t+ 1/(4\lambda)) - 1/4$ and finally $f:\T^{n}\to\T^{n}$ as the skew-product of the form $(h,g)$. 

In the definition of $g$, the addition and the subtraction of the constants $ 1/(4\lambda)$ and $1/4$  are to obtain $f(0) = 0$ and $\frac{\partial}{\partial t} g (x,t) = \lambda $ in a neighborhood  
$B\times I\subset \T^{n}$ of $0$. From this we can conclude all the desired properties in the statement of Theorem \ref{Examp}, except that $f$ is conservative. 

Given $a\in\T^1$, denote by $\psi_a:\T^1\to \T^1$ the map $\psi_a(t) = a + \phi(t+ 1/(4\lambda)) - 1/4$.

in order to prove that $f$ is conservative we need to understand the distribution of the preimages of a point. Given  $(y,s)\in\T^{n}$, we have that
\[f^{-1}(y,s) = \bigcup_{x\in h^{-1}(y)} \{(x,t)\in\T^{n}:g(x,t) = s\}.\]

As we said before $h^{-1}(y)$ has $2^m$ points. Fix $y_0\in p^{-1}(h^{-1}(y))$ and let $X_0 =\{y_0 + \frac{a_1}{2}e_1 + \dots + \frac{a_n}{2}e_n\in \R^n:a_i\in\{0,1\} \}$ where $e_1,\dots,e_n$ is the canonical basis of $\R^n$. Then, the natural projection $p:\R^n \to \T^n$ restricted to $X_0$ is a bijection onto $h^{-1}(y)$. Given $x\in h^{-1}(y)$, take $a_1,\dots,a_n \in\{0,1\}$ such that $x = p(y_0 + \frac{a_1}{2}e_1 + \dots \frac{a_n}{2}e_n)$. If $y_0 = (y^0_1,\dots,y^0_n)$ and $x = (x_1,\dots,x_n)$, then 
\[x_1+\dots + x_n= y^0_1+\dots+y^0_n +  \frac{a_1}{2} + \dots +\frac{a_n}{2} \mod 1.\] 
Define $a =  y^0_1+\dots+y^0_n$ and observe that 
\[\frac{a_1}{2} + \dots +\frac{a_n}{2} \mod 1 = \left\{
	\begin{array}{ll}
			 1/2 & \mbox{if }\#\{a_i:a_i = 1\}\mbox{ is odd} \\
			0 &  \mbox{if }\#\{a_i:a_i = 1\}\mbox{ is even}.
		\end{array}
	\right.\]

From this we conclude that the map
\[
h^{-1}(y) \ni (x_1,\dots,x_n) \mapsto  x_1 +  \dots + x_n \mod 
1,\]
 has $2$ possible values: $a$ and $a + 1/2$. In particular, each one is achieved by $2^m/2$ points of $h^{-1}(y)$. In order to understand the distribution of the preimages along the fiber we only need to study two maps, $\psi_a$ and $\psi_{a+1/2}$. 

Let us call $I\subset \T^1$ the interval where $\frac{\partial}{\partial t} \psi_a = \lambda$. By construction $|\psi_a(I)| = 1/2$ and therefore $\psi_a(I)\cap\psi_{a +1/2}(I)= \emptyset$. This means that, given $s\in \T^1$, unless $s$ lies on the boundary of $\psi_a(I)$, then either $s\in \psi_a(I)$ or $s\in \psi_{a+1/2}(I)$. If $s\in \psi_{a}(I)$, then there exists $t_0\in \psi_{a}^{-1}(s)$ such that $\frac{\partial}{\partial t}\psi_{a}(t_0) = \lambda$ and for the remaining $k-1$ points in $t\in\psi^{-1}_a(s)$ we have $\frac{\partial}{\partial t}\psi_{a}(t)=\eta$. On the other hand, since $s\notin \psi_{a+1/2}(I)$, we have $\frac{\partial}{\partial t}\psi_{a+1/2}(t) = \eta$ for all $t\in \psi_{a+1/2}^{-1}(s)$.

Note that, since $f$ is a skew-product, we have $J(f,(x,t)) = J(h,x)J(g_x,t)$. Consequently, on the full volume set where $J(f,(x,t))$ is well defined, it can attain one out of two possible values, $2^m\lambda$ or $2^m\eta$. We now put everything together. By Equation \ref{EqJac},  to prove that $f$ is conservative is equivalent to check that
\[\frac{2^n}{2} \left(\frac{1}{2^n \lambda} + (k-1)\frac{1}{2^n\eta}\right) + \frac{2^n}{2}k \frac{1}{2^n\eta}= 1.\] 
This is can be simplified to 
\[\frac{1}{\lambda} + \frac{2k-1}{\eta} = 2,\]
and replacing $\eta$ by its value $ \frac{(2k-1)\lambda}{2\lambda - 1}$ we verify the previous equation and therefore the map $f$ is conservative.
\end{proof}

%%%%%%%%%%%%%%%%%%%%%%%%%%%%%%%%%%%%%%%%%%%%%%%%%%%%%%%%%%%%%%%%%%%%%%%%%%%%%%%%%%%%%%%%%%%%%%%%%%%%%%%%%%%%%%%%%%%%%%%%%%%%%%%%%%%%%%%%%%%%%%%%%%%%%%%%%%%%%%%%%%%%
%%%%%%%%%%%%%%%%%%%%%%%%%%%%%%%%%%%%%%%%%%%%%%%%%%%%%%%                  SECTION 3                       %%%%%%%%%%%%%%%%%%%%%%%%%%%%%%%%%%%%%%%%%%%%%%%%%%%%%%%%%%%
%%%%%%%%%%%%%%%%%%%%%%%%%%%%%%%%%%%%%%%%%%%%%%%%%%%%%%%%%%%%%%%%%%%%%%%%%%%%%%%%%%%%%%%%%%%%%%%%%%%%%%%%%%%%%%%%%%%%%%%%%%%%%%%%%%%%%%%%%%%%%%%%%%%%%%%%%%%%%%%%%%%%
\newpage

\section{Fundamental Group of Invariant Regions}                                                        

Through out this section $f:\T^n\to \T^n$ will be a volume-preserving endomorphism. 

\begin{defn} We say that an open subset $U\subset\T^n$ is an invariant region for $f:\T^n\to \T^n$  if $f^{-1}(U) = U$. 
\end{defn}

The motivation for this definition is the observation that if $U$ is an invariant region for $f$, then $U$ together with the restriction of $f$ to $U$ is itself a covering space.

\begin{prop}\label{NotTran}
If $f:\T^n\to\T^n$ is a conservative endomorphism, then the following are equivalent:
\begin{itemize}
\item $f$ is not transitive,
\item there exist $U,V\subset \T^n$ invariant regions for $\T^n$, such that $U$ is equal to the interior of $\T^n -V$.
\end{itemize}  
\end{prop}

For a proof of this proposition check further Proposition $3.2$ in \cite{An15}.  Our objective now is to have a more comfortable framework. This means to suppose that $U$ is connected.

\begin{lema} \label{periodiccomp} Let $f:\T^n\to\T^n$ be a conservative non-invertible endomorphism and suppose that $U$ is an invariant region. If $U_0$ is a connected component of $U$, then there exists $m\geq 1$ such that $U_0$ is an invariant region for $f^m$. 
\end{lema}

See Lemma $3.9$ in \cite{An15} for a proof. 

 The proof of Theorem \ref{SP2} will be by contradiction. Suppose that $f$ is not transitive. Then, by Proposition \ref{NotTran}, there are disjoint invariant regions $U$ and $V$ for $f$. By Lemma \ref{periodiccomp} each connected component of $U$ and $V$ is periodic. This means that there exist $m_1,m_2 \geq 1$, and connected components $U_0$ and $V_0$ of $U$ and $V$ respectively, such that $U_0$ is an invariant region for $f^{m_1}$ and $V_0$ is an invariant region for $f^{m_2}$. In particular, taking $m=m_1 m_2$, we have that both $U_0$ and $V_0$ are invariant regions for $f^m$. Since we are assuming that $f$ is not transitive, neither is $f^m$. Now, cleary the hypotheses in Theorem \ref{SP2} also hold for $f^m$. Therefore it suffices to consider the case in which both $U$ and $V$ are connected.

\begin{lema}\label{ipnt}
Let $f:\T^n\to\T^n$ be a conservative non-invertible endomorphism and suppose that $U$ is an invariant region. If $i:U\to \T^n$ is the inclusion and  $i_\htag :\pi_1(U)\to \pi_1(\T^n)$  is the group morphism induced by $i$, then $i_\htag$ is not trivial.   
\end{lema}

See Lemma $3.6$ in \cite{An15} for a proof. From now on, we will assume that if $U$ is an invariant region, it is also connected.

Let us set the following notation. Given $f:\T^n\to\T^n$ and $U$ an invariant region, take $i:U\to \T^n$ to be the inclusion. Let $p:\R^n\to \T^n$ be the natural projection. A lift of $f$ is a homeomorphism $\hf:\R^n\to \R^n$ such that $f\circ p = p \circ \hf$. We write $\hu= p^{-1}(U)$. The composition of a lift of $f$ with a translation by a vector in $\Z^2$ is again a lift of $f$. Consequently, we can (and do) choose a lift $\hf$ of $f$ such that $\hu$ is invariant for $\hf$.

\pagebreak

\begin{lema} \label{LemDia}
Let $f:\T^n\to\T^n$ be an endomorphism and $U$ an invariant region. Take $i:U\to \T^n$, $p:\R^n\to \T^n$, $\hu\subset \R^n$ and $\hf:\R^n\to\R^n$ as before. Then
\begin{enumerate}
\item the diagram
\[\begin{CD}
0 @>>> \pi_1(\hu) @>p_\htag >> \pi_1(U) @>i_\htag >> \pi_1(\T^n)\\
@. @VV\hf_\htag V @VV(f_{|U})_\htag V  @VVf_\htag V \\
0 @>>> \pi_1(\hu) @>p_\htag >> \pi_1(U) @>i_\htag >> \pi_1(\T^n)
\end{CD}\]
is commutative, and
\item the sequence $0\rightarrow \pi_1(\hu) \stackrel{p_\htag }{\rightarrow} \pi_1(U) \stackrel{i_\htag }{\rightarrow} \pi_1(\T^n)$ is exact.
\end{enumerate}
\end{lema}

\begin{proof}
The commutativity of the first square follows from the fact that $p\circ \hf = f\circ p$. The commutativity of the second square follows from the fact that $i:U\to \T^n$ is the inclusion.

 Let us prove the exactness. Observe that the injectivity of $p_\htag$ holds because $p$ is a covering map. In order to prove that $\Ker(i_\htag) = \Ima(p_\htag)$, fix a point $\hat{x}\in \hu$ and $x= p(\hat{x})$. Given $\gamma:[0,1]\to U$ a continuous curve such that $\gamma(0)=\gamma(1)= x$ observe that $[\gamma]\in \Ker(i_\htag )$ if $\gamma$ is homotopic to the constant curve $x$ in $\T^n$. This happens if and only if the lift $\hat{\gamma}$ of $\gamma$ on $\hat{x}$ verifies $\hat{\gamma}(0)=\hat{\gamma}(1)$. Therefore $ \hat{\gamma}$ is a closed curve in $\hu$ which represents an element of $\pi_1(\hu)$, and $p_\htag ([\hat{\gamma}]) = [\gamma]$.   
\end{proof}

\begin{rem}\label{RemSur}
In the previous situation,  $i_\htag :\pi_1(U) \to \pi_1(\T^n)$ might not be surjective. However, if we replace $\pi_1(\T^n)$ with $i_\htag (\pi_1(U))$, we obtain that the diagram
\[\begin{CD}
0 @>>> \pi_1(\hu) @>p_\htag >> \pi_1(U) @>i_\htag >> i_\htag (\pi_1(U)) @>>>0\\
@. @VV\hf_\htag V @VV(f_{|U})_\htag V  @VVf_\htag V \\
0 @>>> \pi_1(\hu) @>p_\htag >> \pi_1(U) @>i_\htag >> i_\htag (\pi_1(U)) @>>>0
\end{CD}\]
is commutative and the sequence $0\rightarrow \pi_1(\hu) \stackrel{p_\htag }{\rightarrow} \pi_1(U) \stackrel{i_\htag }{\rightarrow} i_\htag (\pi_1(U)) \rightarrow 0$ is exact.  Thereore, $i_\htag(\pi_1(U))$ is isomorphic to the quotient group $\pi_1(U) / \Ima(p_\htag)$.
\end{rem}

\begin{defn}
Given a group morphism $\phi:H\to G$ we define the degree of $\phi$ by $deg(\phi)= [G:\phi(H)]$, this is the number of elements in the quotient $G/\phi(H)$. 
\end{defn}

\begin{rem} \label{DegDet}
In the previous definition, if $H=G = \Z^n$, then $\phi$ can be represented by a matrix $A_\phi \in M_n(\Z)$. In such case, $deg(\phi)=|det(A_\phi)|$.
\end{rem}

We recall a classical result from the theory of covering spaces. 
\begin{teo}\label{degc}
Let $X$ and $Y$ be path connected topological spaces and $g:X\to Y$ a covering map.  Then, the number of sheets of $g$ is equal to $deg(g_\htag )$, where $g_\htag :\pi_1(X) \to\pi_1(Y)$ is the group morphism induced by $g$. 
\end{teo}
A proof of this result can be found in \cite{Ha02}.

The following lemma, in combination with Theorem \ref{degc} and Lemma \ref{LemDia}, will be the main ingredient in the proof of Theorem \ref{SP2}. It is a purely algebraic result:

\pagebreak

\begin{lema}\label{LemAlg}
Let $H, G$ and $K$ be groups, and let $\alpha: H\to G$, $\beta:G\to K$, $\phi:H\to H$, $\psi:G\to G$ and $\nu:K\to K$ be group morphisms such that:
\begin{itemize}
\item $\phi$ is an isomorphism.
\item the sequence $H\stackrel{\alpha}{\rightarrow} G \stackrel{\beta}{\rightarrow} K \to 0$ is exact.
\item the diagram
\[
\begin{CD}
H @>\alpha >> G @>\beta>> K @>>> 0 \\
@VV\phi V @VV\psi V  @VV\nu V \\
H @>\alpha >> G @>\beta>> K @>>> 0 \\
\end{CD}
\]
is commutative.
\end{itemize}
Then, $\deg(\psi)=\deg(\nu)$. 
\end{lema}

\begin{proof}
Take $N = \Ima(\alpha)= \Ker(\beta)\triangleleft G$. Then, $\psi (N) = \psi(\alpha(H))= \alpha(\phi(H)) = \alpha(H) = N$ because $\phi$ is an isomorphism. This allow us to define the morphism $\tilde{\psi}:G/N\to G/N$ by  $\tilde{\psi}(gN)= \psi(g)N$. Take also $\tilde{\beta}:G/N\to K$ defined by $\tilde{\beta}(gN) = \beta(g)$. Since $\beta$ is surjective and $N = \Ker(\beta)$,  $\tilde{\beta}$ is an isomorphism. Since $\nu \circ \beta = \beta \circ \psi$,  we have $\nu \circ \tilde{\beta} = \tilde{\beta} \circ \tilde{\psi}$. This means that the following diagram is commutative: 
\[
\begin{CD}
G/N @>\tilde{\beta}>> K \\
@VV\tilde{\psi}V @VV \nu V\\
G/N @>\tilde{\beta}>> K \\
\end{CD}
\]
Since $\tilde{\beta}$ is an isomorphism, we have $\deg(\nu) = \deg(\tilde{\psi})$. It remains then to prove that $ \deg(\psi)=\deg(\tilde{\psi})$. This is, we need to see that $[G:\psi(G)] = [G/N:\tilde{\psi}(G/N)]$. Since $N = \psi (N)$, we have that $N \triangleleft \psi(G)< G$. Now, $\tilde{\psi}(G/N) =\{\psi(g)N: g\in G\}= \psi(G)/N$, where the first equality comes from the definition of $\tilde{\psi}$ and the second holds because $N<\psi (G)$. So our problem is reduced to prove that $[G:\psi(G)] = [G/N:\psi(G)/N)]$. In order to do this, we define the map $\eta: G/\psi(G) \to (G/N)/(\psi(G)/N)$ by $\eta(g\psi(G))= (gN)(\psi(G)/N)$. It is well defined and bijective. Hence $\deg(\tilde{\psi}) = \deg(\psi)$. 
\end{proof}

Let us fix now a convenient notation. Given a subset $B\subset \R^n$ we define $<B>\subset \R^n$ as the subspace induced by $B$.

The following is the main lemma of this paper:

\begin{lema}\label{LemPrin}
Let $f:\T^n\to\T^n$ be a volume-preserving endomorphism and $U$ an invariant region. If $S= <i_\htag(\pi_1(U))>$, then $|\det (A_{f|S})|=|\det(A_f)|$.
\end{lema}
\begin{proof}
We observed in Remark \ref{RemSur} that the diagram 
\[\begin{CD}
\pi_1(\hu) @>p_\htag >> \pi_1(U) @>i_\htag >>  i_\htag (\pi_1(U)) @>>> 0  \\
@VV\hf_\htag V @VV(f_{|U})_\htag V  @VV(f_\htag )_{|i_\htag (\pi_1(U))}V \\
\pi_1(\hu) @>p_\htag >> \pi_1(U) @>i_\htag >> i_\htag (\pi_1(U)) @>>> 0
\end{CD}\]
is commutative and the sequence $\pi_1(\hu) \stackrel{p_\htag }{\rightarrow} \pi_1(U) \stackrel{i_\htag }{\rightarrow} i_\htag (\pi_1(U)) \to 0$ is exact.  Since $\hf$ is a homeomorphism, $\hf_\htag $ is an isomorphism and we can apply Lemma \ref{LemAlg}. We have the following equation:

\begin{equation}\resizebox{.8\hsize}{!}{$|\det(A_{f|S})| \stackrel{(1)}{=} \deg((f_\htag )_{|i_\htag (\pi_1(U))}) \stackrel{(2)}{=} \deg((f_{|U})_\htag ) \stackrel{(3)}{=} \deg((f_{|U}) )  \stackrel{(4)}{=}\deg(f) \stackrel{(5)}{=} \deg(f_\htag) \stackrel{(6)}{=}  |\det(A_f)|,$}
\end{equation}
 
where $(1)$ and $(6)$ holds by Remark \ref{DegDet}, (2) holds by Lemma \ref{LemAlg} and because by Lemma \ref{ipnt} $i_\htag(\pi_1(U))$ is not trivial, $(3)$ and $(5)$ by Theorem \ref{degc} and $(4)$ because  $\deg((f_{|U}) )$ is the number of preimages of any point for the map $f_{|U}$, since $U$ is an invariant region this number coincides with the number of preimages of $f$ which is $\deg(f)$.
\end{proof}

\begin{rem}
If we remove the volume-preserving hypothesis from Lemma \ref{LemPrin}, we obtain that $i_\htag (\pi_1(U)) = \{0\}$ or $ |\det (A_{f|S})|=|\det(A_f)|$.
\end{rem}

%%%%%%%%%%%%%%%%%%%%%%%%%%%%%%%%%%%%%%%%%%%%%%%%%%%%%%%%%%%%%%%%%%%%%%%%%%%%%%%%%%%%%%%%%%%%%%%%%%%%%%%%%%%%%%%%%%%%%%%%%%%%%%%%%%%%%%%%%%%%%%%%%%%%%%%%%%%%%%%%%%%%
%%%%%%%%%%%%%%%%%%%%%%%%%%%%%%%%%%%%%%%%%%%%%%%%%%%%%%%                  SECTION 4                       %%%%%%%%%%%%%%%%%%%%%%%%%%%%%%%%%%%%%%%%%%%%%%%%%%%%%%%%%%%
%%%%%%%%%%%%%%%%%%%%%%%%%%%%%%%%%%%%%%%%%%%%%%%%%%%%%%%%%%%%%%%%%%%%%%%%%%%%%%%%%%%%%%%%%%%%%%%%%%%%%%%%%%%%%%%%%%%%%%%%%%%%%%%%%%%%%%%%%%%%%%%%%%%%%%%%%%%%%%%%%%%%
\newpage

\section{Invariant subspaces of an Integer Matrix}

The objective of this section is to prove the following proposition:

\begin{prop}\label{PropDet}
Given $A\in M_n(\Z)$ and $\{0\}\subsetneq S\subsetneq  \R^n$ an invariant subspace by $A$. If $\det(A_{|S})\in \Z-\{0\}$, then $\det(A_{|S})$ divides $\det(A)$. In particular, $|\det(A_{|S})|\leq |\det(A)|$. 
\end{prop}

With this in mind, we start by showing that if $\lambda$ is a rational eigenvalue of an integer matrix, then $\lambda$ is an integer. Indeed this is a direct consequence of the well known Rational Root Theorem in elementary algebra. We include it for completeness and because it serves as a warm-up for the proof of Proposition \ref{PropDet}.

\begin{lema}\label{LemEig}
Given $A\in M_n(\Z)$ if $\lambda\in \Q$ is an eigenvalue of $A$, then $\lambda \in \Z$.
\end{lema}

\begin{proof}
If $\chi_A$ is the characteristic polynomial of $A$, then all the coefficients of $\chi_A$ belong to $\Z$ and moreover $\chi_A$ is monic. Suppose that $\chi_A(t)= (-1)^n t^n + \sum_{i=0}^{n-1} a_i t^i$ and take $p,q\in \Z$ coprimes, with $q\neq 0$ such that $\chi_A(\frac{p}{q}) = 0$, then $0 = \frac{p^n}{q^n} + \frac{r}{q^{n-1}}$ for some $r\in \Z$. If $r=0$, then $p=0$ and we are done. If $r\neq 0$, then $-q r= p^n$. Since we took $q$ and $p$ coprime, the later equation implies that $q = \pm 1$ and we conclude. 
\end{proof}

The following lemma extends the previous lemma to invariant subspaces.
 
\begin{lema} \label{LemDet}
Given $A\in M_n(\Z)$ and $\{0\}\subsetneq S\subsetneq  \R^n$ an invariant subspace by $A$. If $\det(A|S)\in \Q$, then $\det(A|S)\in \Z$. 
\end{lema}

\begin{proof}
Given $1\leq m \leq n$ we define the $m$ exterior power of $\R^n$ by $V_m =\overbrace{\R^n\otimes\dots\otimes \R^n}^{m\ times}$. In $V_m$ we define the linear map $A_m:V_m\to V_m$ by $A_m(v_1\otimes \dots\otimes v_m) = A(v_1)\otimes\dots \otimes A(v_m)$. If $\{e_1,\dots,e_n\}$ is the canonical basis in $\R^n$, then $\{e_{i_1}\otimes \dots \otimes e_{i_m}:1\leq i_1<\dots<i_m\leq n \}$ is a basis for $V_m$. Each $A_m$ can be represented by a matrix with respect to this basis. Since $A\in M_n(\Z)$, these matrices have integer coefficients.   Notice that if $S \subset \R^n$ is a subspace invariant under $A$, then $\det(A_{|S})$ is an eigenvalue of $A_m$ where $m= \dim(S)$. Hence, applying Lemma \ref{LemEig} to the (integer) matrix of $A_m$, we conclude that $\det(A_{| S})$ is an integer.
\end{proof}

The next lemma will be the final piece to prove Proposition \ref{PropDet}.

\begin{lema}
Given $A\in M_n(\Z)$ and $\{0\}\subsetneq S\subsetneq  \R^n$ an invariant subspace by $A$. Then, there exists $W$ an invariant subspace by $A$ such that $\det(A)=\det(A_{|S})\det(A_{|W})$.
\end{lema}

\begin{proof}
Let us suppose that $A$ is diagonalizable. In that case there is a basis of $\R^n$ consisting of eigenvectors $\{v_1,\dots, v_n\}$ of $A$. Since $S$ is invariant under $A$, there exist $1\leq i_1<\dots<i_m\leq n$ such that $S = <\{v_{i_1},\dots,v_{i_m}\}>$, where $m=\dim(S)$. Therefore, if we take $W=<\{v_i:i \neq i_j\ \forall j = 1,\dots, m\}>$, then $W$ is invariant under $A$. Observe that $\det(A_{|S}) = \prod_{j=1}^m \lambda_{i_j}$ and $\det(A_{|W}) = \prod_{i\notin\{i_1,\dots,i_k\}} \lambda_i$, where $\lambda_i$ is the eigenvalue associated to $v_i$ for all $1\leq i \leq n$. Since $\det(A)= \prod_{i=1}^{n}\lambda_i$, we have  $\det(A)=\det(A_{|S})\det(A_{|W})$.

We are going to address the case when $A$ is not diagonalizable and it has no complex eigenvalues. The other case will be discussed later.  We will now take the real Jordan form associated to $A$. Let us briefly recall what this is. If $A$ is diagonalizable, it means that there exist a diagonal matrix $D$ associated to $A$ and a basis $\B$ (formed by eigenvectors) such that the linear map associated to $A$ is represented by $D$ in the basis $\B$. When $A$ is not diagonalizable, we have an almost diagonal matrix $J$ associated to $A$ and a basis $\B$ such that the linear map associated to $A$ is represented by $J$ in the basis $\B$.  

With  $J$ the real Jordan form of $A$ we are going to decompose our invariant subspace $S$ in small invariant subspaces $S_l$, where each one will be a subspace associated to a Jordan Block $J_l$. For each $S_l$ we are going to build an invariant subspace  $W_l$ and then $W=\oplus_l W_l$ will verify the desired equation.
 
Given a Jordan block $J_l$ we consider $\B_l$ the elements of the Jordan basis $\B$ associated to $J_l$. That is, $\B_l$ is the set  $\{v^l_1,\dots,v^l_{k_l}\}\subset \B$ such that $A(v^l_j) = \lambda_l v^l_j + v^l_{j+1}$ if $1 \leq j < k_l$ and   $A(v^l_{k_l}) = \lambda_l v^l_{k_l}$ where $\lambda_l$ is the eigenvalue associated to $J_l$. Consider  the subspace $V_l$ induced by $\B_l$. Note that $S_l = V_l\cap S$ is invariant under $A$ and therefore is going to be the subspace induced by $\{v^l_{m_l},\dots,v^l_{k_l}\}$ where $k_l - m_l = \dim (S_l)$. Observe that we cannot build $W$ as before because the induced space by $\{v^l_1,\dots, v^l_{m_l-1}\}$ is not invariant. However, if we define $W_l$ by $W_l=<\{v^l_{k_l - m_l},\dots, v^l_{k_l}\}>$, then it is invariant and the equation $\det(A_{|V_l})=\det(A_{|S_l})\det(A_{|W_l})$ holds. Now  if we define $W = \oplus_l W_l$, we have that $W$ is invariant by $A$ and the equation $\det(A)=\det(A_{|S})\det(A_{|W})$ also holds.

When there are complex eigenvalues, since the characteristic polynomial of $A$ has real coefficients, these necessarily come in pairs of complex conjugates.  For each of this pairs corresponds a two dimensional subspace on which $A$ acts as a composition of a rotation with a homothecy. Its Jordan block is of the form  $\left( \begin{smallmatrix}
a & b \\
-b & a
\end{smallmatrix}\right)$ in $\R^2$. From this, if we take $\B_l$ to be the elements of the Jordan basis $\B$ associated to the Jordan block $J_l$ where the eigenvalue $\lambda_l$ is complex, we have the following:
\[\B_l = \{v^{l,1}_1,v^{l,2}_1,\dots,v^{l,1}_{k_l},v^{l,2}_{k_l}\},\]
\[A(v^{l,1}_j)= a_l v^{l,1}_j - b_l v^{l,2}_j + v^{l,1}_{j+1}\ if\ 1\leq j< k_l, \]
\[A(v^{l,2}_j)= b_l v^{l,1}_j + a_l v^{l,2}_j + v^{l,2}_{j+1}\ if\ 1\leq j< k_l,\]
\[A(v^{l,1}_{k_l})= a_l v^{l,1}_{k_l} - b_l v^{l,2}_{k_l}, \]
and
\[A(v^{l,2}_{k_l})= b_l v^{l,1}_{k_l} + a_l v^{l,2}_{k_l}.\]
In this case, if $V_l$ is the subspace induced by $\B_l$ and $S_l= S\cap V_l$, we have that $S_l$ is the subspace induced by $\{v^{l,1}_{m_l},v^{l,2}_{m_l} \dots,v^{l,1}_{k_l},v^{l,2}_{k_l}\}$ where $2(k_l - m_l) = dim(S_l)$. We then build analogously $W_l$ and $W$.

\end{proof}

\begin{proof}[Proof of Proposition \ref{PropDet}]
Given $\{0\}\subsetneq S\subsetneq  \R^n$ such that $\det(A_{|S})\in \Z$, by the previous Lemma take $W$ invariant by $A$ which verifies $\det(A_{|S})\det(A_{|W})=\det(A)$. Since $\det(A_{|S}) \in \Z$ and $\det(A)\in \Z$, then $\det(A_{|W})\in \Q$. By Lemma \ref{LemDet}, $\det(A_{|W})\in \Z$.
\end{proof}

%%%%%%%%%%%%%%%%%%%%%%%%%%%%%%%%%%%%%%%%%%%%%%%%%%%%%%%%%%%%%%%%%%%%%%%%%%%%%%%%%%%%%%%%%%%%%%%%%%%%%%%%%%%%%%%%%%%%%%%%%%%%%%%%%%%%%%%%%%%%%%%%%%%%%%%%%%%%%%%%%%%%
%%%%%%%%%%%%%%%%%%%%%%%%%%%%%%%%%%%%%%%%%%%%%%%%%%%%%%%                  SECTION 5                       %%%%%%%%%%%%%%%%%%%%%%%%%%%%%%%%%%%%%%%%%%%%%%%%%%%%%%%%%%%
%%%%%%%%%%%%%%%%%%%%%%%%%%%%%%%%%%%%%%%%%%%%%%%%%%%%%%%%%%%%%%%%%%%%%%%%%%%%%%%%%%%%%%%%%%%%%%%%%%%%%%%%%%%%%%%%%%%%%%%%%%%%%%%%%%%%%%%%%%%%%%%%%%%%%%%%%%%%%%%%%%%%
\newpage

\section{Skew-Products of codimension 1}

The main objective of this section is to prove Theorem \ref{SP2} and Corollaries \ref{CorDom1} and \ref{CorDom2}.

Let us define $\hat r_1:\R^n \to \R^{n-1}$ and $ r_1:\T^n\to \T^{n-1}$ by 
\[\hat r_1(x_1,\dots, x_n) = (x_1,\dots, x_{n-1})\ \forall (x_1,\dots,x_n)\in \R^n,\]
and 
\[ r_1(x_1,\dots, x_n) = (x_1,\dots, x_{n-1})\ \forall (x_1,\dots,x_n)\in \T^n.\]
 The skew-product structure of $f$ implies that $r_1 \circ f = h \circ r_1$. In particular, we can take lifts $\hf:\R^n\to \R^n$ and $\hat h:\R^{n-1}\to \R^{n-1}$  of $f$ and $h$ such that $\hat r_1 \circ  \hf = \hat h \circ \hat r_1$.

The following two properties verified by a lift of $f$ and the linear map of $f$ come from classical arguments (check further \cite{Wa70}):

\begin{prop} \label{LinProp}Let $f:\T^n\to \T^n$ be a continuous map, $\hat f:\R^n \to \R^n$ be a lift of $f$, $f_\htag :\Z^n \to \Z^n$ the induced action by $f$ in the fundamental group of $\T^n$ and $A_f\in M_n(\Z)$ the matrix that represents $f_\htag $. Then, we have the following equations:
\begin{equation}\label{lift1} \hat f (x+v) = \hat f (x) + f_\htag (v)\ \forall x\in \R^n,\ \forall v\in \Z^n,\end{equation}
\begin{equation}\label{lift2} \exists L_0 > 0\text{ such that }d(\hat f (x), A_f(x)) \leq L_0\ \forall x\in \R^n. \end{equation}
\end{prop}

For the sake of completeness we give a proof of these statements.

\begin{proof}
Let $p:\R^n\to \T^n$ be the natural projection and define $\hat \alpha:[0,1]\to \R^n$ and $\alpha:[0,1]\to \T^n$ by $\hat \alpha(t) = x+tv$ and $\alpha(t) = p(\alpha(t))$. In particular, $\hat f \circ \hat \alpha$ is curve which starts at $\hat f (x)$ and ends at $\hat f(x+v)$. On the other hand, by definition, $\alpha$ is a loop and $[\alpha] = v$. Now  $p\circ \hat f \circ \hat \alpha = f \circ p \circ \hat \alpha = f \circ \alpha$ and therefore $\hat f \circ \hat \alpha$ is a lift of the curve $f \circ \alpha$. By definition, $f_\htag (v) = f_\htag ([\alpha]) = [f\circ \alpha]$ and then $\hat f \circ \hat \alpha$ is a lift which begins at $\hat f(x)$ and ends at $\hat f(x) + f_\htag (v)$ obtaining that $\hat f(x+v) = \hat f(x) + f_\htag (v)$.

For the second equation define $\left\lfloor x \right\rfloor \in \Z^n $ such that $x - \left\lfloor x \right\rfloor \in[0,1)^n$. Using Equation \ref{lift1}, we have:
\[\hat f (x) = \hat f(x - \left\lfloor x \right\rfloor + \left\lfloor x \right\rfloor ) = \hat f(x - \left\lfloor x \right\rfloor) + f_\htag (\left\lfloor x \right\rfloor),\] 
then
\[\hat f(x) - A_f(x) = \hat f(x - \left\lfloor x \right\rfloor) + f_\htag (\left\lfloor x \right\rfloor) -A_f( x -\left\lfloor x \right\rfloor) - A_f(\left\lfloor x \right\rfloor).\]
Since $ f_\htag (\left\lfloor x \right\rfloor) = A_f(\left\lfloor x \right\rfloor)$, we obtain $\hat f(x) - A_f(x) = \hat f(x - \left\lfloor x \right\rfloor) -A_f( x -\left\lfloor x \right\rfloor)$. This implies that $\Ima(\hat f - A_f) = \Ima((\hat f  - A_f)_{|[0,1]^n})$. Using the compactness of $[0,1]^n$, we conclude the proposition.
\end{proof}

We return to the skew-product structure with the following construction: As we said before $\pm \deg(g)$ is an eigenvalue of $A_f$ and $e_n = (0,\dots,0,1)$ is the eigenvector associated to it (the eigenvalue is $\deg(g)$ if $f$ preserves the orientation on the fiber and $-\deg(g)$ if $f$ reverses it). Let  $\{v_1,\dots, v_{n-1},e_n\}$ be a Jordan basis for $A_f$.  If  $J_n$ is the Jordan block associated to $e_n$ and if $dim(J_n)=1$,  then  $P_0= <\{v_1,\dots,v_{n-1}\}>$  is a hyperplane invariant by  $A_f$,   transverse to $e_n$.  In particular, the transverse condition implies that $\hat r_{1|P_0}$ is a linear isomorphism. This is the only place where we use the hypothesis $\dim(J_n)=1$. It guarantees the existence of $P_0$.

\begin{rem}\label{RemOri}
Let $\lambda_n$ be the eigenvalue of $A_f$ associated to the eigenvector $e_n$. Observe that if $\lambda_n < 0$,  then $\lambda_n^2 >0$ is the eigenvalue associated to $e_n$ under the map $A_f^2 = A_{f^2}$. If $f$ is not transitive, neither is $f^2$. We may therefore assume in what follows that $f$ preserves the orientation in the fibers. 
\end{rem}

Given two parallel hyperplanes $P_1,P_2\subset \R^n$ we call $[P_1,P_2]\subset \R^n$ the connected set which has $P_1\cup P_2$ as its boundary and for which $\R^n\cap [P_1,P_2]^c$ has two connected components. 

\begin{lema}
There exists $k_1\leq k_2\in \Z$ such that 
\[\hf^{-1}([P_0+k_1 e_n,P_0 + k_2 e_n])\subset  [P_0+k_1 e_n,P_0 + k_2 e_n].\]
\end{lema}

\begin{proof}
By Proposition \ref{LinProp}, take $L_0 > 0$ such that $d(\hf(x),A_f(x))\leq L_0 \ \forall x\in \R^n$. Since $P_0$ is invariant for $A_f$, we have that 
\[\hf (P_0) \subset [P_0-L_0 e_n, P_0+L_0 e_n].\]
Combining this and the Equation \ref{lift1} from Proposition \ref{LinProp}, we obtain 
\[\hf(P_0 + ke_n)\subset [P_0+ (\deg(g) k - L_0)e_n, P_0 + (\deg(g) k + L_0)e_n].\]
Take $k_1,k_2\in \Z$ such that $\deg(g) k_2 - L_0 \geq k_2$ and $\deg(g) k_1 + L_0 \leq k_1$. We can find such $k_1$ and $k_2$ because $\deg(g)\geq 2$. In particular, $k_1 \leq 0\leq k_2$.
For these $k_1$ and $k_2$ we have that
\[\hf([P_0 + k_1 e_n,P_0 + k_2 e_n])\supset [P_0 +k_1 e_n,P_0 + k_2 e_n].\]
Since $\hf$ is a homeomorphism, if we apply $\hf^{-1}$ to the previous equation, we conclude the lemma.
\end{proof}

We define the set $S_0$ by 
\[S_0= \bigcap_{k\in\N} \hf^{-k}([P_0+k_1 e_n,P_0 + k_2 e_n]).\]
Let us define $r_0:S_0\to \R^{n-1}$ by $r_0(s) = \hat r_1(s)\ \forall s\in S_0$.

Given a continuous curve $\alpha:[0,1]\to \R^n$ such that $\alpha(1) = \alpha(0) + v$ with $v\in\Z^n$ we define its periodic continuation as $\alpha_\infty:\R \to \R^n$ by $\alpha_\infty(t) = v \left\lfloor t\right\rfloor + \alpha (t - \left\lfloor t \right\rfloor)$, where $\left\lfloor t \right\rfloor$ is the integer part of $t$.

\begin{lema}\label{PropS0}
The set $S_0$ verifies the following:
\begin{enumerate} 
\item $\hf(S_0) = S_0$.
\item $r_0(S_0) = \R^{n-1}$.
\item $\R^n\cap S_0^c$ has two connected components.
\item Given $\alpha:[0,1]\to \R^n$ such that $\alpha(1) = \alpha(0) + v$ with $v$ transverse to $P_0$, then $S_0\cap \Ima(\alpha_\infty)\neq\emptyset$.
\item $r_0^{-1}(x)$ is a connected set for every $ x\in \R^{n-1}$ and is therefore either a point or an interval. 
\item $r_0$ is a semi-conjugacy between $\hf_{|S_0}$ and $\hat h$. This is $\hat h \circ r_0 = r_0 \circ \hf_{|S_0}$
\end{enumerate}
\end{lema}

\begin{proof}
\begin{enumerate}
\item By definition of $S_0$.
\item Given $x\in \R^{n-1}$ we have that  
\[\emptyset \neq \bigcap_{k\in\N} \hf^k(\{\hat h^{-k}(x) \}\times [k_1,k_2])\subset S_0\cap \hat r_1^{-1}(x).\]
\item $\R^n\cap \hf^{-k}( [P_0+k_1 e_n,P_0 + k_2 e_n])^c$ has two connected components because $\hf$ is a homeomorphism, and $\R^n\cap [P_0+k_1 e_n,P_0 + k_2 e_n]^c$ has two connected components. Then, this property is verified by $S_0$.
\item Given such a curve $\alpha$,  $\alpha_\infty$ intersects both connected components of $\R^n \cap  [P_0+k_1 e_n,P_0 + k_2 e_n]^c$. Therefore $\alpha$ intersects both connected components of $\R^n\cap S_0^c$. Since $\alpha_\infty$ is continuous, it intersects $S_0$ in some point.   
\item If two points project to the same point, then the dynamics of both remain between the two hyperplanes. Since $\hf$ is a skew-product, the whole segment remains between the two planes.
\item This is because $\hat r_1$ is a semi-conjugacy between $\hf$ and $\hat h$.
\end{enumerate}
\end{proof}

\begin{lema}\label{PropS}
There exists $S\subset \R^n$ which verifies:
\begin{enumerate}
\item $\hf(S) = S$.
\item If we define $r:S\to \R^{n-1}$ by $r = \hat r_1|S$, then $r(S) = \R^{n-1}$.
\item Given $\alpha:[0,1]\to \R^n$ such that $\alpha(1) = \alpha(0) + v$ with $v$ transverse to $P_0$, then $S\cap \Ima(\alpha_\infty)\neq\emptyset$.
\item $r^{-1}(x)$ is a connected set $\forall x\in \R^{n-1}$, therefore either a point or an interval. 
\item $r$ is a semi-conjugacy between $\hf_{|S}$ and $\hat h$. This is $\hat h \circ r = r \circ \hf_{|S}$
\item The interior of $S$ is empty.
\end{enumerate}

\end{lema}

\begin{proof}
The set $S_0$ in Proposition \ref{PropS0} has all the above properties except possibly for the last one. If it does, we take $S= S_0$ and we are done. If this is not the case, for each $x\in \R^{n-1}$ we define $a(x)\leq b(x)\in \R$ such that  $r_0^{-1}(x) = \{x\}\times [a(x),b(x)]$. We build now the following two sets 
\[ \min(S_0) = \left(\bigcup_{x\in r_0(Int(S_0))}\{(x,a(x))\} \right)\cup\left( \bigcup_{x\notin r_0(Int(S_0))} r_0^{-1}(x)\right), \] 
and 
\[ \max(S_0) =\left(\bigcup_{x\in r_0(Int(S_0))}\{(x,b(x))\}\right) \cup \left(\bigcup_{x\notin r_0(Int(S_0))} r_0^{-1}(x)\right).\] 
These sets are well defined by item $2$ and $5$ in Lemma \ref{PropS0}. In particular, the boundary of $S_0$ verifies
\[\partial S_0 = \min(S_0)\cup \max(S_0).\]
Since $f$ preserves the orientation in the fibers, we have that $\hf (\max(S_0)) = \max(S_0)$ and \linebreak $\hf(\min(S_0))= \min(S_0)$. Now $\max(S_0)$ and $\min(S_0)$ verifiy all the properties of Lemma $\ref{PropS0}$ and $int(\max(S_0))=\emptyset$. We define then $S$ as $\max(S_0)$.   
\end{proof}

We are now in condition to prove Theorem \ref{SP2}.

\begin{proof}[Proof of Theorem \ref{SP2}]

Let us assume that $f$ is not transitive. Let $U$ and $V$ be as in Proposition \ref{NotTran} and write  $\hu=p^{-1}(U)$ and $\hv= p^{-1}(V)$. Let  $i:U\to \T^n$ be the inclusion and denote by $P_1$ the subspace $ <i_\htag(\pi_1(U))>$ of $ \R^n$. Suppose that  $i_\htag(\pi_1(U))$ has no element transverse to $P_0$, this means $P_1$ is an $A_f$ invariant subspace of $P_0$. Since $\hat r_{1|P_0}:P_0\to \R^{n-1}$ is a linear isomorphism which conjugates $A_{f|P_0}$ and $A_h$, then $\hat r_1(P_1)$ is an invariant subspace of $A_h$ and $|\det(A_{h|\hat r_1(P_1)})| = |\det(A_{f|P_1})|$. By Lemma \ref{LemPrin}, $|\det(A_{f|P_1})|= |\det(A_f)|$, and since 
\[A_f= \left( \begin{smallmatrix}
A_h & 0 \\
* & \deg(g) 
\end{smallmatrix}\right),\]
we have $ |\det(A_f)| = \deg(g)|\det(A_h)|$. Therefore $|\det(A_{h|\hat r_1(P_1)})|= \deg(g)|\det(A_h)|$. Since $\deg(g)\geq 2$, we have $|\det(A_{h|\hat r_1(P_1)})|> |\det(A_h)|$ which contradicts Proposition \ref{PropDet}.

We have proved that $i_\htag(\pi_1(U))$ has an element transverse to $P_0$. Therefore there exists $\alpha:[0,1]\to \hu$ such that $\alpha(1) = \alpha(0) + v$ with $v$ transverse to $P_0$. By property 3 in Lemma \ref{PropS}, $\hu\cap S\neq \emptyset$. Analogously $\hv\cap S\neq \emptyset$. Let us call these intersections $U_S$ and $V_S$. Since $h$ is transitive, $r(U_S)$ and $r(V_S)$ are open and dense in $\R^{n-1}$. Take $W=int (  r(U_S)\cap r(V_S))\neq \emptyset$. Since $int(S) = \emptyset$, there exists $w\in W$ such that $r^{-1}(w)$ is a point. Such point belongs to $\hu \cap \hv$ which is a contradiction.
\end{proof}

Let us see why Theorem \ref{SP2} implies Corollary \ref{CorDom1}.
\begin{lema}
If $f:\T^n\to \T^n$ is a skew-product endomorphism of the form $f=(h,g)$ such that $|\deg(h)|<|\deg(g)|$, then $\dim(J_n)=1$.
\end{lema}
\begin{proof}
By a simple computation we have:
\[A_f= \left( \begin{smallmatrix}
A_h & 0 \\
* & \pm \deg(g) 
\end{smallmatrix}\right).\]
  
If $\chi_{A_f}$ and $\chi_{A_h}$ are the characteristic polynomials of $A_f$ and $A_h$ respectively, then $\chi_{A_f}(t)=- \chi_{A_h}(t)(t- \pm \deg(g))$. This implies that the eigenvalues of $A_f$  are $\pm \deg(g)$ and the eigenvalues of $A_h$. By Proposition \ref{PropDet}, $\pm \deg(g)$ can not be an eigenvalue of $A_h$ and therefore $\dim(J_n)=1$. 
\end{proof}

Analogously, let us see why Theorem \ref{SP2} implies Corollary \ref{CorDom2}.

\begin{lema}
If $h:\T^n\to \T^n$ is an endomorphism such that $|A_h v|> |\deg(g)| |v|$ $\forall v\in\R^n -\{0\}$, then $\dim(J_n)=1$.
\end{lema}

\begin{proof}
By the arguments of the previous lemma, we just need to show that $\pm \deg(g)$ is not an eigenvalue of $A_h$. If it were, then there would exists $v\in \R^n - \{0\}$ such that $A_h v = \pm \deg(g) v$. This contradicts our hypothesis.
\end{proof}


\begin{thebibliography}{99}

\bibitem[A]{An15}	
M. Andersson, \emph{Transitivity of conservative toral endomorphisms}. Nonlinearity, Volume 29, Number 3, 1047.
	
\bibitem[AH]{AH94} N. Aoki, K. Hiraide, \emph{Topological theory of dynamical systems}, North-Holland Mathematical Library, vol. 52, North-Holland Publishing Co., Amsterdam, 1994, Recent Advances.
	
\bibitem[H]{Ha02}
A. Hatcher, Algebraic Topology. Cambridge University Press. (2002) 

\bibitem [LP] {LiPu12}
C. Lizana and E. Pujals, Robust Transitivity for Endomorphisms. Ergodic Theory and Dynamical Systems, Volume 33, Issue 4
 (2013), 1082-1114.
	
		
\bibitem [R] {Ra16}	
W. Ranter, Transitive Endomorphisms with Critical Points. arxiv:1608.06921.


\bibitem[W]{Wa70}	
P. Walters, Anosov Diffeomorphisms are topologically stable, Topology 9 (1970), 71-78.


\end{thebibliography}
\end{document}